\documentclass{article}
\usepackage{amsmath,amssymb,amsthm}

\theoremstyle{definition}
\newtheorem{definition}{Definition}[section]
\newtheorem{notation}[definition]{Notation}
\newtheorem{example}[definition]{Example}

\theoremstyle{plain}
\newtheorem{theorem}[definition]{Theorem}
\newtheorem{lemma}[definition]{Lemma}
\newtheorem{proposition}[definition]{Proposition}
\newtheorem{corollary}[definition]{Corollary}
\newtheorem{remark}[definition]{Remark}

\newcommand{\beq}{\begin{equation}}
\newcommand{\eeq}{\end{equation}}
\newcommand{\bdfn}{\begin{definition}}
\newcommand{\edfn}{\end{definition}}
\newcommand{\bthm}{\begin{theorem}}
\newcommand{\ethm}{\end{theorem}}
\newcommand{\bprop}{\begin{proposition}}
\newcommand{\eprop}{\end{proposition}}
\newcommand{\bcor}{\begin{corollary}}
\newcommand{\ecor}{\end{corollary}}
\newcommand{\blem}{\begin{lemma}}
\newcommand{\elem}{\end{lemma}}
\newcommand{\bex}{\begin{example}}
\newcommand{\eex}{\end{example}}
\newcommand{\bxc}{\begin{exercise}}
\newcommand{\exc}{\end{exercise}}
\newcommand{\bntn}{\begin{notation}}
\newcommand{\entn}{\end{notation}}

\newcommand{\be}{\begin{enumerate}}
\newcommand{\ee}{\end{enumerate}}
\newcommand{\bce}{\begin{center}}
\newcommand{\ece}{\end{center}}
\newcommand{\bi}{\begin{itemize}}
\newcommand{\ei}{\end{itemize}}

\newcommand{\bt}{\begin{tabular}}
\newcommand{\et}{\end{tabular}}

\newcommand{\ra}{\rightarrow}

\newcommand{\Ra}{\Rightarrow}

\newcommand{\se}{\subseteq}
\newcommand{\ol}{\overline}

\newcommand{\nin}{\in\!\!\!\!\!/}

\newcommand{\si}{\wedge}
\newcommand{\sau}{\vee}
\newcommand{\ba}{\begin{array}} 
\newcommand{\ea}{\end{array}}

\def\N{{\mathbb N}}

\newcommand {\bua} {\begin{eqnarray*}}
\newcommand {\eua} {\end {eqnarray*}}

\newcommand{\ds}{\displaystyle}

\begin{document}

\date{}

\title{Maximal residuated lattices with lifting Boolean center}
\author{ George Georgescu${}^1$, Lauren\c{t}iu Leu\c stean$^{2,3}$  and Claudia Mure\c{s}an${}^1$\\[0.2cm]
\footnotesize ${}^1$ 
Faculty of Mathematics and Computer Science, University of Bucharest, \\
\footnotesize Academiei 14, RO 010014, Bucharest, Romania\\
\footnotesize E-mails: ggeorgescu@rdslink.com, c.muresan@yahoo.com \\[0.1cm]
\footnotesize ${}^2$ Department of Mathematics, Technische Universit\" at Darmstadt,\\
\footnotesize Schlossgartenstrasse 7, 64289 Darmstadt, Germany\\[0.1cm]
\footnotesize${}^3$ Institute of Mathematics "Simion Stoilow'' of the Romanian Academy, \\
\footnotesize Calea Grivi\c tei 21, P.O. Box 1-462, Bucharest, Romania\\[0.1cm]
\footnotesize E-mail: leustean@mathematik.tu-darmstadt.de
}
\maketitle

\begin{abstract}
\noindent In this paper we define, inspired by ring theory, the class of maximal resi-duated lattices with lifting Boolean center and prove a structure theorem for them: any maximal residuated lattice with lifting Boolean center is isomorphic to a finite direct product of local residuated lattices.

\vspace*{10pt} 
\end{abstract}

\section{Introduction}

\noindent A\,  {\em commutative integral residuated bounded lattice}\, is an algebraic  structure $(A,\vee ,\wedge ,\odot ,\rightarrow ,0,1)$ such that $(A,\vee ,\wedge ,0,1)$ is a bounded lattice, $(A,\odot ,1)$ is a commutative monoid and, for all $a,b,c\in A$, 
$$a\leq b\rightarrow c \text{ if and only if } a\odot b\leq c.$$

Commutative integral residuated bounded lattices have been studied extensively and include important classes of algebras such as BL-algebras, introduced by ${\rm H\acute{a}jek}$ as the algebraic counterpart of his Basic Logic \cite{haj},  and MV-algebras, the algebraic setting for \L ukasiewicz propositional logic (we refer to the monograph \cite{cdom} for a detailed treatment of MV-algebras). Since in this paper we work only with commutative integral residuated bounded lattices, we shall call them simply {\em residuated lattices}. In order to simplify the notation, a residuated lattice $(A,\vee ,\wedge ,\odot ,\rightarrow ,0,1)$ will be referred by its support set $A$. The {\em Boolean center of $A$}, denoted $B(A)$, is the set of all complemented elements of the  bounded lattice $(A,\vee ,\wedge ,0,1)$.

The main purpose of this paper is to define the class of {\em maximal residuated lattices with lifting Boolean center } and to prove a structure theorem for them. 

The inspiration for defining this class of residuated lattices comes from ring theory. Maximal rings are an important class of commutative rings with unit; we refer to \cite{brandal} for a book treatment. The idea of {\em lifting idempotents}, due to Nicholson \cite{nic}, turns out to be very useful in studying different classes of rings. 

If $A$ is a residuated lattice, $\{a_i\}_{i\in I}\se A$ and $\{F_i\}_{i\in I}$ is a family of filters of $A$, then $A$ is maximal iff, given a family of congruences $\{x\equiv a_i(mod\, F_i)\}_{i\in I}$ of $A$, being able to find a solution for any finite subset of these congruences implies one can find a solution for all the congruences. We refer to Section \ref{maximal} for the formal definition.

Similar notions were developed for distributive lattices \cite{ge-1991}, MV-algebras \cite{figele} and BL-algebras \cite{leo}.  
Obviously, residuated lattices with a finite number of filters are maximal; hence, finite and simple residuated lattices are maximal. The converse is not true. An example of a maximal MV-algebra with an infinite number of ideals is given in \cite[Proposition 9]{figele}.

A residuated lattice $A$ is said to have {\em lifting Boolean center} iff for every $e\in B(A/{\rm Rad}(A))$ there exists a  $f\in B(A)$ such that $e=f/{\rm Rad}(A)$. Here ${\rm Rad}(A)$ is the intersection of all maximal filters of $A$.

The main result of the paper is the following (see Theorem \ref{main-theorem}).\\[0.2cm]
{\bf Theorem.} Any maximal residuated lattice with lifting Boolean center is isomorphic to a finite direct product of local residuated lattices.\\[0.2cm]

This structure theorem corresponds in the setting of residuated lattices to Zelinsky's theorem for maximal rings \cite{Zel53}, \cite[Theorem 2.6]{brandal}. In fact, we prove even a stronger result, namely that a residuated lattice $A$ with lifting Boolean center is maximal if and only if it is isomorphic to a finite direct product of some special residuated lattices, determined by elements in $B(A)$ (see Theorem \ref{char-maximal}).

\section{Definitions and basic properties}\label{prelim1}

\noindent We refer the reader to \cite{GalKowJipOno} for basic results in the theory of residuated lattices. In the following, we only present the material needed in the remainder of the paper. 

We shall denote with ${\cal RL}$ the variety of residuated lattices and ${\bf RL}$ the ca\-te\-go\-ry having as objects nontrivial residuated lattices and morphisms of re\-si\-du\-a\-ted lattices as morphisms. We recall that by a {\em residuated lattice} we mean in fact a {\em commutative integral residuated bounded lattice}.

Let $A$ be a residuated lattice.  We use the notation $L(A)$ for the bounded lattice $(A,\vee ,\wedge ,0,1)$. For all $a,b\in A$, let us define
$$\neg \,a:=a\rightarrow 0, \quad a\leftrightarrow b:=(a\rightarrow b)\wedge (b\rightarrow a).$$

The following lemma collects some useful properties (see for example \cite{GalKowJipOno}).

\blem\label{useful-identities}
For all $a,b,c,d\in A$,
\be
\item $\,\,$ $\neg \, 0=1$, $\neg \, 1=0$;
\item\label{(1)} $\,\,$ $a=1\rightarrow a$ and $1=a\rightarrow 1$;
\item\label{(3)} $\,\,$ $a\leq b$ iff $a\rightarrow b=1$ and $\,\,$ $a\leq \neg \, b$ iff $a\odot b=0$;
\item\label{(-1)} $\,\,$ if $a\leq b$ and $c\leq d$ then $a\odot c\leq b\odot d$;
\item\label{(4)} $\,\,$ if $a\leq b$ then $c\rightarrow a\leq c\rightarrow b$ and $b\rightarrow c\leq a\rightarrow c$;
\item $\,\,$ $a\odot b\leq a\wedge b$;
\item $\,\,$ $a\le b\ra a$;
\item $\,\,$ $a\odot 0=0$, $\,\,0\rightarrow a=1$ and $\,\,a\leftrightarrow 0=\neg \, a$;
\item $\,\,$ $a\leq \neg \, \neg \, a$ and $\,\,\neg \, \neg \, \neg \, a=\neg \, a$;
\item\label{comutodotvee} $\,\,$ $a\odot (b\vee c)=(a\odot b)\vee (a\odot c)$;
\item\label{(2)} $\,\,$ $(a\vee b)\rightarrow c=(a\rightarrow c)\wedge (b\rightarrow c)$;
\item\label{(9)} $\,\,$ $a\rightarrow b\leq (c\rightarrow a)\rightarrow (c\rightarrow b)$;
\item\label{(8)} $\,\,$ $a\rightarrow (b\rightarrow c)=b\rightarrow (a\rightarrow c)$.
\ee
\elem

For all $a\in A$, we define $a^{0}=1$ and $a^n=a^{n-1}\odot a$ for all $n\in \N^{*}$.  The {\em order} of $a\in A$, in symbols $ord(a)$,  is the smallest $n\in\N$ such that  $a^n=0$. If no such $n$ exists, then $ord(a)=\infty$. An element $a\in A$ is called: {\em nilpotent} iff $ord(a)$ is finite; a {\em unity} iff $\neg(a^n)$ is nilpotent for all $n\in\N$; {\em finite} iff both $a$ and $\neg a$ are nilpotent. 

A {\em filter of $A$} is a nonempty set $F\subseteq A$ such that, for all $a,b\in A$,\\(i) $a,b\in F$ implies $a\odot b\in F$;\\(ii) $a\in F$ and $a\leq b$ imply $b\in F$.\\[0.1cm] A filter $F$ of $A$ is {\em proper} iff $F\neq A$. We shall denote by ${\cal F}(A)$ the set of filters of $A$.

A proper filter $P$ of $A$ is called {\em prime} iff $a\vee b\in P  \text{ implies }  a\in P \text{ or } b\in P$ for all $a,b\in A$.
The set of prime filters of $A$ is denoted by ${\rm Spec}(A)$.

A proper filter $M$ of $A$ is called {\em maximal} iff it is not
contained in any other proper filter. We denote by ${\rm Max}(A)$ \index{$Max(A)$} the set of maximal filters of $A$. An immediate application of Zorn's lemma is the fact that any proper filter of $A$ can be extended to a maximal filter. As a consequence, ${\rm Max}(A)\ne\emptyset$ for any nontrivial residuated lattice $A$.

Let  $X\subseteq A$. The filter of $A$ generated by $X$ will be denoted by $<X>$. We have that $<\emptyset>=\{1\}$ and for $X\neq \emptyset$,
\begin{center}$<X>=\{a\in A\mid x_1\odot\cdots\odot x_n\leq a$ for some $n\in\N^{*}$ and  $x_1,\cdots , x_n\in X\}$. 
\end{center}
For any $a\in A$, $<a>$ denotes the principal filter of $A$ generated
by $\{a\}$. Then $<a>=\{b\in A\mid a^n\leq b$ for some
$n\in\N^{*}\}$.

The following results are standard and they are relatively easy to prove; for example, by following the proofs of the corresponding results for BL-algebras from \cite[Chapter 1]{leo}. 

\bprop\label{filtre-latice-completa} 
$({\cal F}(A), \subseteq)$ is a complete lattice. For every family 
$\{F_i\}_{i\in I}$ of filters of $A$, we have that
\begin{eqnarray*}
\bigwedge_{i\in I}F_i  =  \bigcap_{i\in I} F_i, \quad\bigvee_{i\in I}F_i =  <\bigcup_{i\in I} F_i>. 
\end{eqnarray*}
\eprop

The {\em radical } of $A$, denoted by ${\rm Rad}(A)$, is the intersection of all maximal filters of $A$, when $A$ is a nontrivial residuated lattice $A$. If $A=\{0\}$ is trivial, then $Rad(A)=\{0\}$ by definition. 

\bprop\label{Rad-prop}

Let $A$ be a residuated lattice. Then
\be
\item\label{Rad-unity} ${\rm Rad}(A)=\{a\in A\mid a \text{~is a unity}\}$;
\item\label{Rad-subalg} ${\rm Rad}(B)=B\cap {\rm Rad}(A)$ for any  subalgebra $B$ of $A$;
\item\label{Rad-product} ${\rm Rad}\left(\prod _{i\in I}A_{i}\right)=\prod _{i\in I}{\rm Rad}(A_{i})$ for any family $\{A_{i}|i\in I\}$ of residuated lattices. 
\ee
\eprop
\begin{proof}
(i) See \cite[Lemma 4.1]{hoh}. \,\, (ii) and (iii) are easy consequences of (i).\end{proof}

A residuated lattice $A$ is said to be {\em local} iff $A$ has exactly one maximal filter. A local residuated lattice $A$ is called {\em perfect} iff for all $a\in A$, $ord(a)<\infty $ if and only if $ord(\neg a)<\infty $.

\begin{proposition}{\rm \cite{ciu}}
$A$ is local if and only if  $D(A):=\{a\in A \mid ord(a)=\infty \}$ is the unique maximal filter of $A$.
\end{proposition}

The following lemma will be useful in Section \ref{maximal}.

\blem\label{filtre-max-1}
Define $\ds F_M:=\{a\in A\mid \mbox{\rm the set } \{M\in {\rm Max}(A)\mid a\nin\, M\}\, \mbox {\rm is finite} \}$.
Then $F_M$ is a filter of $A$ and for any finite subset $\{M_1,\ldots, M_n\}$ of ${\rm Max}(A)$,
\[\displaystyle\bigcap \{M\mid M\in {\rm Max}(A)-\{M_1,\ldots, M_n\} \}\se F_M. \]
\elem
\begin{proof}
We have that $1\in F_M$, since $\{M\in {\rm Max}(A)\mid 1\nin\, M\}=\emptyset$. If $a,b\in A$, then 
\[\{M\!\in\! {\rm Max(A)}\!\mid\! a\odot b\nin M\}\!=\!\{M\!\in\! {\rm  Max(A)}\!\mid\!a\nin M\}\cup\{M\!\in\!{\rm Max}(A)\!\mid\!b\nin M\},\]
hence $a,b\in F_M$ implies $a\odot b\in F_M$. If $a\le b$, then 
\[\{M\in {\rm Max}(A)\mid b\nin\,M\}\se \{M\in {\rm Max}(A)\mid a\nin\,M\},\]
hence $a\in F_M$ implies $b\in F_M$.\\
If $a\nin\,F_M$, then the set $\{M\in Max(A)\mid a\nin\,M\}$ is infinite, hence there is $M\in {\rm Max}(A)-\{M_1,\ldots, M_n\}$ such that $a\nin M$. Thus, $a\nin\, \bigcap\{M\mid M\in Max(A)-\{M_1,\ldots, M_n\}\}$.\end{proof}

\bprop\label{subalgebra-cardinal-filters}
Let $B$ be a  residuated lattice and $A$ be a subalgebra of $B$. Then 
$$|{\cal F(A)}|\le |{\cal F(B)}|, \,\, |Spec(A)|\le|Spec(B)| \text{ and } |Max(A)|\le |Max(B)|.$$
\eprop
\begin{proof}
Just follow the proof of the corresponding result for BL-algebras \cite[Proposition 1.2.25]{leo}.\end{proof}

If $h:A\rightarrow B$ is a morphism of residuated lattices, then the {\em kernel of $h$} is the set ${\rm Ker}(h):=\{a\in A \mid h(a)=1\}$. Obviously, $h$ is injective iff ${\rm Ker}(h)=\{1\}$. 

\bprop \label{hom}
Let $h:A\ra B$ be a morphism of residuated lattices. Then the following properties hold:
\be
\item\label{inverse-image-h-F}  for any (proper, prime, maximal) filter $F$ of $B$, the set $h^{-1}(F)=\{a\in A\mid h(a)\in F\}$ is a (proper, prime, maximal) filter of $A$; thus, in particular, $Ker(h)$ is a proper filter of $A$;
\item\label{surjective-hF} if $h$ is surjective and $F$ is a filter of $A$, then $h(F)$ is a filter of B;
\item\label{surjective-hF-maximal} if $h$ is surjective and $M$ is a maximal filter of $A$ such that $h(M)$ is proper, then $h(M)$ is a maximal filter of B;
\item\label{hom-cardinal-image} if $h$ is surjective, then  \[|{\cal F}(B)|\le |{\cal F}(A)|,\,\, |{\rm Spec}(B)|\le |{\rm Spec}(A)| \text{ and } |{\rm Max}(B)|\le |{\rm Max}(A)|.\]
\ee
\eprop

With any filter $F$ of $A$ we can associate a congruence relation $\equiv(mod\,F)$ on $A$
by defining
\[a\equiv b({\rm mod}\ F)\,\,\mbox{\rm  if and only if  } \,\,a\leftrightarrow b\in F.\]
For any $a\in A$, let $a/F$  be the equivalence class $a/_{\equiv({\rm mod}\ F)}$. If we denote
by $A/F$ the quotient set
$A/_{\equiv( mod\,F)}$, then $A/F$ becomes a residuated lattice with the operations induced from those of $A$.

\blem \label{prop-cong}
Let $F$ be a filter of $A$ and $a,b\in A$. Then
\be
\item\label{aF=0Fsau1F} $a/F=1/F$ iff $a\in F$ and $a/F=0/F$ iff $\neg a\in F$;
\item\label{aF-leq-bF} $a/F\le b/F$ iff $a\ra b\in F$;
\item\label{prop-cong-Fproper} if $F$ is proper and $a/F=0/F$, then $a\nin F$.
\ee
\elem

The following proposition follows from a general result in universal algebra \cite{bur}. A proof for this particular case is similar to the proof of \cite[Proposition 1]{figele}.

\bprop{\bf (Chinese Remainder Theorem)}\label{chinese}\\
Let $n\in{\rm I\! N}^{*}$ and $F_1,\ldots, F_n$ be filters of the residuated lattice $A$ such that $F_i\sau F_j=A$ for every $i\ne j,\,\, i,j\in \ol{1,n}$. Then, for every $a_1,\ldots a_n\in A$, there exists an $a\in A$ such that $a\equiv a_i(mod\,F_i)$ for all $i\in \ol{1,n}$. 
\eprop

For any filter $F$ of $A$, let us denote by $p_F$ the  quotient map from $A$ onto $A/F$, defined by $p_F(a)=a/F$ for any $a\in A$. Then $F=Ker(p_F)$. For simplicity, we shall use the notation $G/F$ for $p_F(G)$. 

\blem\label{lemma-p_F}
Let $F,G$ be filters of $A$ such that $F\se G$. Then
\be
\item\label{p_F-aF-in-GF} for all $a\in A$, $a/F\in G/F$ iff $a\in G$; 
\item $G$ is proper iff $G/F$ is a proper filter of $A/F$;
\ee
\elem

\bprop\label{prop-p_F}
Let $F,G$ be filters of $A$ such that  $F\se G$. Then
\be
\item the map $p_F$ is an inclusion-preserving bijective correspondence between the filters of $A$ containing $F$ and the filters of $A/F$; the inverse map is also inclusion-preserving; 
\item\label{maximal-proper-A-AF} $p_F$ maps  the set of proper (maximal) filters of $A$ containing $F$ onto the set of proper (maximal) filters of $A/F$;
\item\label{phi-F-G} the map $\phi:A/F\to A/G, \,\,\phi(a/F)=a/G$ is a well-defined surjective morphism of residuated lattices; $\phi$ is an isomorphism if and only if $F=G$;
\item\label{second-iso} the map $(A/F)/(G/F)\to A/G, \,\,\,(a/F)/(G/F)\mapsto a/G$ is an isomorphism of residuated lattices.
\ee
\eprop

As an immediate application of Proposition \ref{prop-p_F}.(\ref{maximal-proper-A-AF}) and Lemma \ref{lemma-p_F}, we get the following.

\bprop\label{Max-A-AF-RadA}
Let $A$ be a nontrivial residuated lattice and $F$ be a proper filter of $A$. 
\be
\item Then $|{\rm Max}(A/F)|\le |{\rm Max}(A)|$. 
\item Assume that $F\se {\rm Rad}(A)$. Then 
\bce $|{\rm Max}(A/F)|=|{\rm Max}(A)|$ and ${\rm Rad}(A/F)={\rm Rad}(A)/F$.
\ece In particular, $A$ is local if and only if $A/F$ is local. 
\ee
\eprop

\bprop\label{product-AF}
Let $\{A_{i}\mid i\in I\}$ be a family of residuated lattices and $F_{i}$ be a filter of $A_i$ for every $i\in I$. Then $F:=\prod _{i\in I}F_{i}$  is a filter of $A:=\prod _{i\in I}A_{i}$ and $A/F=\prod _{i\in I}A_i/F_i$.  Moreover, if $p_F:A\to A/F,\, p_i:A_i\to A_i/F_i\, (i\in I)$ are the quotient maps, then $p_F=\prod _{i\in I}p_i$.
\eprop

Let $B(A)$ be the {\em Boolean center of $A$}, that is the set of all complemented e\-le\-ments of the lattice $L(A)$. The following lemma collects some useful properties of $B(A)$.

\bprop\cite{BusPic06,GalKowJipOno}\label{BA-prop}
\be
\item\label{BA-neg-e} $B(A)$ is a Boolean subalgebra of $L(A)$, $\neg\,e$ is the unique complement of $e\in B(A)$ and $\neg\neg\, e=e$.
\item\label{e-idempotent} For any $e\in B(A)$, $e\odot e=e$ and $<e>=\{a\in A\mid e\leq a\}$.
\item\label{BA-e-f} For all $e,f\in B(A)$, $e\odot f=e\si f\in B(A)$, $e\ra f=\neg\, e\sau f\in B(A)$ and $e\leftrightarrow f=(e\ra f)\si (f\ra e)\in B(A)$.
\item\label{BA-RadA=1} $B(A)\cap {\rm Rad}(A)=\{1\}$.
\item\label{B-product} $ B\left(\prod _{i\in I}A_{i}\right)=\prod _{i\in I}B(A_{i})$ for any family $\{A_{i}\mid i\in I\}$ of residuated lattices.
\ee
\eprop

\begin{lemma}
For every $e,f\in B(A)$ and $a,b\in A$, we have:
\be
\item\label{(0)} if $e\leq a$ then $\neg\, e\rightarrow a=a$;
\item\label{(10)} $e\rightarrow a=e\rightarrow (e\rightarrow a)$;
\item\label{(7)} $e\rightarrow (a\rightarrow b)=(e\rightarrow a)\rightarrow (e\rightarrow b)$;
\item\label{evx} $\neg \, e\rightarrow a=e\vee a$;
\item\label{comutarea} $a\vee (e\wedge f)=(a\vee e)\wedge (a\vee f)$.
\ee
\label{bcomut}
\end{lemma}
\begin{proof}
(\ref{(0)}), (\ref{(10)}), (\ref{(7)}) and (\ref{evx}) were proven by ${\rm K\ddot{u}hr}$ in \cite{kuh} for bounded pseudo-BCK-algebras,  hence for noncommutative residuated lattices. For the sake of completeness, we give the proofs in the commutative case.
\be
\item $a=1\rightarrow a=(e\vee \neg\, e)\rightarrow a=(e\rightarrow a)\wedge (\neg\, e\rightarrow a)=\neg\, e\rightarrow a$, since $e\ra a=1$.
\item 
\bua e\rightarrow a&=&1\rightarrow (e\rightarrow a)=(e\vee \neg\, e)\rightarrow (e\rightarrow a)\\
&=&(e\rightarrow (e\rightarrow a))\wedge (\neg\, e\rightarrow (e\rightarrow a))= e\rightarrow (e\rightarrow a),
\eua
since  $\neg\, e=e\rightarrow 0\leq e\rightarrow a$, so $\neg\, e\rightarrow (e\rightarrow a)=1$. 
\item Since $a\leq e\rightarrow a$, we have that $a\rightarrow b\geq (e\rightarrow a)\rightarrow b$, hence
$$e\rightarrow (a\rightarrow b)\geq e\rightarrow ((e\rightarrow a)\rightarrow b)=(e\rightarrow a)\rightarrow (e\rightarrow b).$$
Furthermore, $a\rightarrow b\leq (e\rightarrow a)\rightarrow(e\rightarrow b)$ implies
\bua e\rightarrow (a\rightarrow b)& \leq & e\rightarrow ((e\rightarrow a)\rightarrow(e\rightarrow b))=(e\rightarrow a)\rightarrow (e\rightarrow (e\rightarrow b))\\
&=&(e\rightarrow a)\rightarrow(e\rightarrow b), \text{ by (ii)}.
\eua
\item  Since $a\le \neg\, e\rightarrow a$ and $e=\neg\, \neg\, e=\neg\, e\rightarrow 0\leq \neg\, e\rightarrow a$, it follows that $e\sau a\le  \neg\, e\rightarrow a$.

Let $u\in A$ such that $u\geq e\vee a$. We get that 
\bua(\neg\, e\rightarrow a)\rightarrow u &\stackrel{(i)}{=}& (\neg\, e\rightarrow a)\rightarrow (\neg\, e\rightarrow u)\stackrel{(iii)}{=}\neg e\rightarrow (a\rightarrow u)\\
&=& \neg\, e\rightarrow 1=1.
\eua
\item It is obvious that $a\vee (e\wedge f)\leq (a\vee e)\wedge (a\vee f)$. Now let $x\in A$ such that $x\leq (a\vee e)\wedge (a\vee f)$. Then, by (\ref{evx}), $x\leq a\vee e=\neg \, e\rightarrow a$ and $x\leq a\vee f=\neg \, f\rightarrow a$, so, by the law of residuation, $x\odot \neg \, e\leq a$ and $x\odot \neg \, f\leq a$. It follows that
$$x\odot \neg \, (e\wedge f)=x\odot (\neg \, e\vee \neg \, f)=(x\odot \neg \, e)\vee (x\odot \neg \, f)\leq a,$$
 hence, $x\leq \neg \, (e\wedge f)\rightarrow a=(e\wedge f)\vee a$, by (iv).
\ee
\end{proof}

With the help of the Boolean center we can define a functor $B$ between the category ${\bf RL}$ of residuated lattices and the category {\bf Bool} of Boolean algebras as follows: for any morphism of residuated lattices $f:A_{1}\rightarrow A_{2}$, $B(f):B(A_{1})\to B(A_{2})$ is the restriction of $f$ to $B(A_{1})$.

For each $x\in A$, let us define the operation  
$$\ra_x:A\times A\ra A, \quad a\ra_x b=x\sau(a\ra b).$$

\bprop \label{B(A)-2}
Let  $F$ be a filter of $A$ and $e\in B(A)$. Then
\be
\item ${\bf <e>}=(<e>, \sau, \si, \odot,\ra_e, e, 1)$ is a residuated lattice;
\item $F\,\cap <e>=\{e\sau a\mid a\in F\}$  and $F\,\cap <e>$ is a filter of ${\bf <e>}$;
\item\label{F-Finte} for all $a,b\in A$,
\begin{eqnarray*}
a\equiv b(mod\,\,F)\,\, \mbox{\rm in } A & \mbox{\rm implies } & a\sau e \equiv b\sau e(mod\,\,F\cap<e>)\,\, \mbox{\rm in } {\bf <e>}.
\end{eqnarray*} 
\ee
\eprop
\begin{proof}
Just follow the proof for BL-algebras from \cite[Proposition 1.4.4]{leo}.\end{proof}

\noindent The following results are standard and they can be proved in a similar manner with the corresponding results for MV-algebras (see \cite[Lemmas 6.4.4, 6.4.5]{cdom}).

\bprop\label{finite-direct-product-1}
Let $\{A_i\}_{i\in I}$ be a nonempty family of residuated lattices and let $A\cong \prod_{i\in I}A_i$. Then there exists a family $\{e_i\}_{i\in I}\subseteq B(A)$  satisfying the following conditions:
\be
\item $\displaystyle \si_{i\in I}e_i=0$;
\item $e_i\vee e_j=1$ whenever $i,j\in I, i\ne j$;
\item each $A_i$ is isomorphic to ${\bf <e_i>}$.
\ee
\eprop

\bprop\label{finite-direct-product-2}
Let $A$ be a residuated lattice, $n\ge 2$ and $e_1,\ldots, e_n\in B(A)$ be such that 
\be
\item  $e_1\si\ldots\si e_n=0$;
\item $e_i\sau e_j=1$ whenever $i,j\in \ol{1,n}, i\ne j$.
\ee
Then $A\cong\displaystyle\prod_{i=1}^n{\bf <e_i>}$.  
\eprop

\section{Finite direct products of residuated Lattices}
\label{products}

\noindent In this section, we shall make a study of prime and maximal filters of finite direct products of residuated lattices, similar to the one done for MV-algebras in \cite{bedile92} or BL-algebras in \cite{leo}. 

In the sequel, $I$ will be an index set, $\{A_i\}_{i\in I}$ a family of nontrivial residuated lattices and $A=\displaystyle\prod_{i\in I}A_i$. For each $i\in I$, let $pr_i:A\ra A_i, \quad pr_i\big((a_k)_{k\in I}\big)=a_i$ be the projections. Obviously, $pr_i$ is a surjective morphism of residuated lattices.
For any $i\in I$, let us denote by $\delta_i$ the element of $A$ defined by 
\[pr_i(\delta_i)=0\, \mbox{\rm and } pr_j(\delta_i)=1\,\, \mbox{\rm  for all } j\in I,\,j\ne i. \]

It is clear that $\delta_i\sau\delta_k=1$ for all $i,k\in I, \,i\ne k$.

\bprop\label{produs-1}
Let $P$ be a prime filter of $A$. Then 
\be
\item $pr_i(P)\ne A_i$ for at most one $i\in I$;
\item if $I$ is finite, then $pr_i(P)\ne A_i$ for exactly one $i\in I$.

\ee
\eprop
\begin{proof}$\,$
\be
\item Suppose that there are $i,k\in I, i\ne k$ such that $pr_i(P)\ne A_i$ and $pr_k(P)\ne A_k$. 
Then $\delta_i\sau\delta_k=1\in P$, so, $\delta_i\in P$ or $\delta_k\in P$, since $P$ is a prime filter. It follows that $0=pr_i(\delta_i)\in pr_i(P)$ or $0=pr_k(\delta_k)\in pr_k(P)$. Since, by Proposition \ref{hom}.(\ref{surjective-hF}), $pr_i(P), pr_k(P)$ are filters, we get that $pr_i(P)=A_i$ or $pr_k(P)=A_k$. That is, we have got a contradiction.
\item Let $I=\{1,\ldots,n\}$. By (i), there exists at most one $i\in I$ such that $pr_i(P)\ne A_i$. Suppose that there is no such $i$, that is $pr_i(P)=A_i$ for all $i\in \ol{1,n}$. It follows that for every $i$ there exists $a_i\in P$ such that $pr_i(a_i)=0$. If we let $a:=a_1\odot \ldots\odot a_n$, we get that $a\in P$ and $pr_i(a)=0$ for all $i\in \ol{1,n}$, so $a=0$. Thus, we have got that $0\in P$, a contradiction.
\ee \end{proof}

If $i\in I$ and $Q$ is a prime filter of $A_i$, $pr_i^{-1}(Q)$ is a prime filter of $A$, by Proposition \ref{hom}.(\ref{inverse-image-h-F}). We call it the {\em filter of A over Q} and we denote it by $Ov(Q)$. Let us define
\bua 
Ov(A)&:=&\{P\se A\mid  P=Ov(Q)\text{ for some } Q\in \displaystyle\bigcup_{i\in I}{\rm Spec}(A_i)\}\\
&=& \bigcup \{p_i^{-1}(Q)\mid i\in I, \,Q\in {\rm Spec}(A_i)\}.
\eua
Then $Ov(A)\ne \emptyset$ and $Ov(A)\se {\rm Spec}(A)$.

\bprop\label{main-finite-product}
Suppose that $I$ is finite  and $P$ is a prime filter of $A$. Let $i\in I$  be unique such that $pr_i(P)\ne A_i$. Then
\be 
\item $pr_i(P)$ is a prime filter of $A_i$ and, if $P$ is maximal, then $pr_i(P)$ is also maximal;
\item $P=Ov(pr_i(P))$.
\ee
\eprop
\begin{proof} Let $I=\{1,\ldots, n\}$.
\be
\item By Proposition \ref{hom}.(\ref{surjective-hF}) and the fact that $pr_i(P)\ne A_i$, we get that $pr_i(P)$ is a proper filter of $A$. Let $x,y\in A_i$ be such that $x\sau y\in pr_i(P)$, so $x\sau y=pr_i(c)$ for some $c\in P$. Let us define $c_1,c_2\in A$ by:
$$pr_i(c_1)=x, pr_i(c_2)=y \text{ and } pr_j(c_1)=pr_j(c_2)=pr_j(c)\quad \text{for all }j\ne i.$$
Since $c_1\sau c_2=c\in P$, we must have $c_1\in P$ or $c_2\in P$, hence $x\in pr_i(P)$ or $y\in pr_i(P)$.

If $P$ is maximal, apply Proposition \ref{hom}.(\ref{surjective-hF-maximal}) to get that $pr_i(P)$ is maximal.
\item $Ov(pr_i(P))=pr_i^{-1}(pr_i(P))\supseteq P$. It remains to prove the converse inclusion. Let $a\in Ov(pr_i(P))$. Since $pr_i(a)\in pr_i(P)$, $pr_i(q_i)=pr_i(a)$ for some $q_i\in P$. For $j\in I$, $j\ne i$, $pr_j(P)=A_j$, so there exists $q_j\in P$ such that $pr_j(q_j)=pr_j(a)$. Let 
$$q:=(\delta_1\sau q_1)\odot\ldots\odot (\delta_n\sau q_n).$$
 Then $q\in P$ and $pr_j(q)=pr_j(q_j)=pr_j(a)$  for all $j\in I$, so $q=a$, hence $a\in P$. 
\ee \end{proof}

We are now ready to prove the main result of this section.

\bthm\label{spec+max-finite-product}
Let $n\in {\rm I\! N}^{*}$, $A_1,\ldots,A_n$ be nontrivial residuated lattices and $A=\displaystyle\prod_{i=1}^n A_i$. Then 
\bua
{\rm Spec}(A) & = & \displaystyle\bigcup_{i=1}^n \{A_1\times\ldots\times A_{i-1}\times Q\times A_{i+1}\times \ldots\times A_n\mid Q\in {\rm Spec}(A_i)\},\\
{\rm Max}(A) & = & \displaystyle\bigcup_{i=1}^n \{A_1\times\ldots\times A_{i-1}\times Q\times A_{i+1}\times \ldots\times A_n\mid Q\in {\rm Max}(A_i)\},
\eua
hence
\bua 
\left|{\rm Spec}(A)\right|  = \displaystyle\sum_{i=1}^n\left|{\rm Spec}(A_i)\right| \quad \text{and}\quad \left|{\rm Max}(A)\right| = \displaystyle\sum_{i=1}^n\left|{\rm Max}(A_i)\right|.
\eua
\ethm
\begin{proof} The first equality is an immediate application of Proposition \ref{main-finite-product}; we get that ${\rm Spec}(A)=Ov(A)$, hence the first equality. 

If $M\in {\rm Max}(A)$ and $i\in \ol{1,n}$ is unique such that $pr_i(M)\ne A_i$, then $pr_i(M)\in {\rm Max}(A_i)$ by Proposition \ref{main-finite-product}, and $M=Ov(pr_i(M))=A_1\times\ldots\times A_{i-1}\times pr_i(M)\times A_{i+1}\times \ldots\times A_n$. Conversely, if  $Q\in {\rm Max}(A_i)$ for some $i\in \ol{1,n}$, then $A_1\times\ldots\times A_{i-1}\times Q\times A_{i+1}\times \ldots\times A_n=pr_i^{-1}(Q)\in {\rm Max}(A)$ by Proposition \ref{hom}.(\ref{inverse-image-h-F}).\end{proof}

\section{Dense elements and lifting Boolean center}\label{dense-lifting}

\noindent Let $A$ be a residuated lattice. An element $a$ of $A$ is said to be {\em dense} iff $\neg \,a=0$. Following \cite{fre}, we denote by $Ds(A)$ the set of the dense elements of $A$. It is easy to see that $Ds(A)$ is a filter of $A$ satisfying $Ds(A)\se {\rm Rad}(A)$ \cite{fre}.


\blem\label{F-se-Ds-prop}
Let $A$ be a residuated lattice and $F\se Ds(A)$ be a filter of $A$. Then
\be 
\item\label{Ds-aF-0F} for all $a,b\in A$, $a/F=0/F$ iff $a=0$ and $a/F\leq \neg \, b/F$ iff $a\leq \neg \, b$;
\item\label{ord-a-aF} for all $a\in A$, $ord(a)=ord(a/F)$;
\item\label{ord-a-aF-infinite}  $p_F\big(\{a\in A \mid ord(a)=\infty \}\big)=\{a/F\mid ord(a/F)=\infty \}$;
\item\label{ord-a-aF-finite} for all $a\in A$, $a$ is finite in $A$ if and only if $a/F$ is finite in $A/F$. Hence,  $p_F(\{a\in A \mid a \text{ is finite}\})=\{a/F\mid a/F \text{ is finite}\}$.
\ee
\elem
\begin{proof}$\,$
\be
\item  By Lemma \ref{prop-cong}.(\ref{aF=0Fsau1F}), $a/F=0/F$ iff $\neg a\in F\se Ds(A)$. Thus, $a/F=0/F$ implies $\neg a\in Ds(A)$, so $\neg \, \neg \, a=0$, that is equivalent to $a=0$, since $a\leq \neg\neg a$. The converse implication is obvious.

$a/F\leq \neg \, b/F$ iff $a\leq \neg \, b$ follows using the above and the fact that in any residuated lattice $x\le\neg \, y$ iff $x\odot y=0$. 
\item By (\ref{Ds-aF-0F}), for all $a\in A$ and all $n\in\N$,  $a^{n}=0$ iff $a^{n}/F=0/F$ iff $(a/F)^{n}=0/F$, hence $ord(A)=ord(a/F)$.
\item $p_F\big(\{a\in A \mid ord(a)\!=\!\infty \}\big)\!\!=\!\!\{a/F \mid ord(a)\!=\!\infty \}\!\stackrel{(\ref{ord-a-aF})}{=}\!\{a/F\mid ord(a/F)\!=\!\infty \}$.
\item follows easily from (\ref{ord-a-aF}).
\ee \end{proof}

\bprop\label{Ds-prop}
Let $A$ be a residuated lattice. Then

\be 
\item\label{Ds-A-F} $Ds(A/F)=Ds(A)/F$ for any filter $F$ of $A$ contained in $Ds(A)$;
\item\label{Ds-subalg} $Ds(B)=B\cap Ds(A)$ for any subalgebra $B$ of $A$;
\item\label{Ds-product} $\displaystyle Ds\left(\prod _{i\in I}A_i\right)=\prod _{i\in I}Ds(A_{i})$ for any family $\{A_{i}\mid i\in I\}$ of residuated lattices.
\ee
\eprop
\begin{proof} (\ref{Ds-A-F}) is an immediate consequence of Lemma \ref{F-se-Ds-prop}.(\ref{Ds-aF-0F}). (\ref{Ds-subalg}), (\ref{Ds-product}) are obvious.\end{proof}

\begin{proposition}\label{Ds-local-perfect}
A residuated lattice $A$ is local (perfect) if and only if $A/Ds(A)$ is local (perfect).
\label{lrlocala}
\end{proposition}
\begin{proof}
Since  $Ds(A)\se {\rm Rad}(A)$, we can apply Proposition \ref{Max-A-AF-RadA} to get that $A$ is local if and only if $A/Ds(A)$ is local. Use Proposition \ref{F-se-Ds-prop}.(\ref{ord-a-aF}) to obtain the other equivalence.\end{proof}

Let $f:A\rightarrow B$ be a morphism of residuated lattices and define
\bua
\overline{f}:A/Ds(A)\rightarrow B/Ds(B), \quad\overline{f}(a/Ds(A))=f(a)/Ds(B).
\eua 
It is easy to see that $\overline{f}$ is a well-defined morphism of residuated lattices and that the diagram below is commutative ($p_{A}$ and $p_{B}$ are the canonical surjections).

\begin{center}
\begin{picture}(170,70)(0,0)

\put(17,53){$A$}
\put(20,50){\vector(0,-1){20}}
\put(4,38){$p_{A}$}
\put(125,38){$p_{B}$}

\put(20,18){$A/Ds(A)$}
\put(26,58){\vector(1,0){90}}
\put(68,61){$f$}

\put(119,53){$B$}
\put(119,18){$B/Ds(B)$}
\put(62,22){\vector(1,0){54}}

\put(122,50){\vector(0,-1){20}}
\put(87,9){$\overline{f}$}
\end{picture}
\end{center}

As a consequence, we can define a (covariant) functor ${\bf T}:{\bf RL}\rightarrow {\bf RL}$  by setting ${\bf  T}(A)=A/Ds(A)$ and ${\bf T}(f)=\ol{f}$.

\begin{proposition}
{\bf T} preserves surjective morphisms, injective morphisms and direct products.
\label{T-functor}
\end{proposition}
\begin{proof}
Let $f:A\rightarrow B$ be a morphism of residuated lattices. If $f$ is surjective, then  $p_{B}\circ f=\overline{f}\circ p_{A}$ is also surjective, hence  $\overline{f}$ is surjective.

Assume now that $f$ is injective and let $a_{1},a_{2}\in A$. Then we have the fo\-l\-lo\-wing sequence of equivalences:  $\overline{f}(a_{1}/Ds(A))=\overline{f}(a_{2}/Ds(A))$ iff $f(a_{1})/Ds(B)=f(a_{2})/Ds(B)$ iff $f(a_{1})\leftrightarrow f(a_{2})\in Ds(B)$ iff $f(a_{1}\leftrightarrow a_{2})\in Ds(B)$ iff $\neg \,f(a_{1}\leftrightarrow a_{2})=0$ iff $f(\neg \,(a_{1}\leftrightarrow a_{2}))=0$ iff (since $f$ is injective) $\neg \,(a_{1}\leftrightarrow a_{2})=0$ iff $a_{1}\leftrightarrow a_{2}\in Ds(A)$ iff $a_{1}/Ds(A)=a_{2}/Ds(A)$. Hence, $\overline{f}$ is injective.

Let $\{A_{i}\mid i\in I\}$ be a family of residuated lattices and $ A=\prod _{i\in I}A_{i}$. By Proposition \ref{Ds-prop}.(\ref{Ds-product}), $ Ds(A)=\prod _{i\in I}Ds(A_{i})$. Apply now Proposition \ref{product-AF}.\end{proof}

A residuated lattice $A$ is said to be {\em radical-dense} iff ${\rm Rad}(A)=Ds(A)$. Let us denote with  rd-${\cal RL}$ the class of radical-dense residuated lattice. This terminology is inspired by \cite{fre}, where a variety ${\cal A}$ is called {\em radical-dense} provided that $A$ is a subvariety of $\cal RL$ and ${\rm Rad}(A)=Ds(A)$.

\begin{proposition}
rd-${\cal RL}$ is  closed with respect to subalgebras and direct products.
\end{proposition}

\begin{proof}
Apply Propositions \ref{Ds-prop}.(\ref{Ds-subalg}) and \ref{Rad-prop}.(\ref{Rad-subalg}) to get closure under subalgebras. For obtaining closure with respect to direct products, use Propositions \ref{Ds-prop}.(\ref{Ds-product}) and \ref{Rad-prop}.(\ref{Rad-product}).
\end{proof}

In the following, let $A$ be a residuated lattice. Since $Ds(A)\se {\rm Rad}(A)$, we can apply Proposition \ref{prop-p_F}.(\ref{phi-F-G}) to get a surjective morphism of residuated lattices $\phi _{A}:A/Ds(A)\to A/{\rm Rad}(A), \,\phi _{A}(a/Ds(A))=a/{\rm Rad}(A)$ that makes the following diagram commutative. 
\begin{center}
\begin{picture}(70,70)(0,0)
\put(2,43){$A$}
\put(12,48){\vector(1,0){38}}
\put(27,51){$p_{A}$}
\put(53,43){$A/Ds(A)$}
\put(56,40){\vector(0,-1){26}}
\put(59,25){$\phi _{A}$}
\put(10,40){\vector(4,-3){40}}
\put(14,20){$r_{A}$}
\put(53,2){$A/{\rm Rad}(A)$}
\end{picture}
\end{center}
In the diagram above, $p_A$ and $r_A$ are the quotient maps. Moreover, $\phi _{A}$ is an isomorphism if and only if $A$ is radical-dense, that is $Ds(A)={\rm Rad}(A)$. 

If $B:{\bf RL}\to {\bf Bool}$ is the functor defined in Section \ref{prelim1},  the diagram above induces the following commutative diagram in the category of Boolean algebras.
\begin{center}
\begin{picture}(70,70)(0,0)
\put(2,43){$B(A)$}
\put(32,48){\vector(1,0){38}}
\put(34,51){$B(p_{A})$}
\put(73,43){$B(A/Ds(A))$}
\put(76,40){\vector(0,-1){26}}
\put(79,25){$B(\phi _{A})$}
\put(30,40){\vector(4,-3){40}}
\put(14,20){$B(r_{A})$}
\put(73,2){$B(A/{\rm Rad}(A))$}
\end{picture}
\end{center}

\begin{lemma}
$B(p_{A})$ and $B(r_{A})$ are injective.
\end{lemma}

\begin{proof}
For all $e,f\in B(A)$,  we have that $B(r_{A})(e)=B(r_{A})(f)$ iff $e/{\rm Rad}(A)=f/{\rm Rad}(A)$ iff $e\leftrightarrow f\in B(A)\cap {\rm Rad}(A)$ iff $e\leftrightarrow f=1$ \big(by Proposition \ref{BA-prop}.(\ref{BA-e-f}), (\ref{BA-RadA=1})\big) iff $e=f$. Hence, $B(r_{A})$ is injective and the fact that $B(p_{A})$ is injective follows from the commutativity of the diagram.\end{proof}

We say that $A$ has {\em lifting Boolean center} iff $B(r_{A})$ is surjective (and hence a Boolean isomorphism). 

\begin{remark}
It is easy to see that the definition given above coincides with the one from the introduction: $A$ has lifting Boolean center if and only if for every $e\in B(A/{\rm Rad}(A))$ there exists an $f\in B(A)$ such that $e=f/{\rm Rad}(A)$.
\end{remark}

The fact that MV-algebras have lifting Boolean center was already proved in \cite[Proposition 5]{figele}. Moreover, using \cite[Lemma 2.7.6]{leo}, we can conclude that $BL$-algebras have lifting Boolean center too.

\bprop\label{rd-li}
Any radical-dense residuated lattice has lifting Boolean center. 
\eprop
\begin{proof}
If $A$ is radical-dense, then $\phi_A$ is an isomorphism, by Proposition \ref{prop-p_F}.(\ref{phi-F-G}), hence $B(\phi_A)$ is an isomorphism of Boolean algebras. This implies obviously the surjectivity of $B(r_A)$.\end{proof}

\bprop\label{LP-prop}
\be
\item If $A$ has lifting Boolean center then $B(\phi _{A})$ is surjective.
\item $A/Ds(A)$ has lifting Boolean center if and only if $B(\phi _A)$ is a Boolean isomorphism.
\item Let $\{A_{i}\mid i\in I\}$ be a family of residuated lattices.  Then $\prod _{i\in I}A_{i}$ has lifting Boolean center if and only if $A_i$ has lifting Boolean center for every $i\in I$.
\ee
\eprop
\begin{proof}
\be
\item follows from the commutative diagram above.
\item For simplicity, we denote $A^\star\!:=A/Ds(A)$ and $H\!:={\rm Rad}(A)/Ds(A)$.  Since $Ds(A)\se {\rm Rad}(A)$, ${\rm Rad}(A^\star)=H$ by Proposition \ref{Max-A-AF-RadA}. Moreover, by applying Proposition \ref{prop-p_F}.(\ref{second-iso}) we obtain an isomorphism 
\[\psi: A^\star/H\to A/{\rm Rad}(A), \quad \psi\left((a/Ds(A))/H\right)=a/{\rm Rad}(A).\]

If $r_{A^\star}:A^\star\to A^\star/{\rm Rad}(A^\star)$ is the quotient map, then $\psi\circ r_{A^\star}=\phi_A$. By applying the functor $B$, we get that $B(\psi)$ is an isomorphism of Boolean algebras such that $B(\psi)\circ B\left(r_{A^\star}\right)=B(\phi_A)$.

It follows that $A^\star$ has lifting Boolean center if and only if $B\left(r_{A^\star}\right)$ is an isomorphism if and only if $B(\phi_A)$ is an isomorphism.
\item Let $A:=\prod _{i\in I}A_{i}$, $r_i:A_i\to A_i/{\rm Rad}(A_i), r_A:A\to A/{\rm Rad}(A)$ be the quotient maps. Then $r_A=\prod _{i\in I}r_i$ by Propositions \ref{product-AF} and \ref{Rad-prop}.(\ref{Rad-product}). Moreover, since $B(A)=\prod _{i\in I}B(A_{i})$, it follows that $B(r_A)=\prod _{i\in I}B(r_i)$. 
\ee \end{proof}

We finish this section with an example of a residuated lattice without lifting Boolean center. 

\begin{example}\cite{GalKowJipOno}
Let $A=\{0,a,b,c,d,1\}$ be the following residuated lattice.

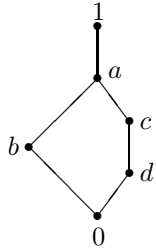
\begin{center}
\begin{picture}(60,100)(0,0)
\put(37,11){\circle*{3}}
\put(35,0){$0$}
\put(37,11){\line(3,4){12}}
\put(49,27){\circle*{3}}
\put(53,24){$d$}
\put(49,27){\line(0,1){20}}
\put(49,47){\circle*{3}}
\put(53,44){$c$}
\put(49,47){\line(-3,4){12}}
\put(37,63){\circle*{3}}
\put(41,63){$a$}
\put(37,11){\line(-1,1){26}}
\put(11,37){\circle*{3}}
\put(3,34){$b$}
\put(11,37){\line(1,1){26}}
\put(37,63){\line(0,1){20}}
\put(37,83){\circle*{3}}
\put(35,85){$1$}
\end{picture}
\end{center}

\begin{center}
\begin{tabular}{cc}
\begin{tabular}{c|cccccc}
$\rightarrow $ & $0$ & $a$ & $b$ & $c$ & $d$ & $1$ \\ \hline
$0$ & $1$ & $1$ & $1$ & $1$ & $1$ & $1$ \\
$a$ & $0$ & $1$ & $b$ & $c$ & $c$ & $1$ \\
$b$ & $c$ & $1$ & $1$ & $c$ & $c$ & $1$ \\
$c$ & $b$ & $1$ & $b$ & $1$ & $a$ & $1$ \\
$d$ & $b$ & $1$ & $b$ & $1$ & $1$ & $1$ \\
$1$ & $0$ & $a$ & $b$ & $c$ & $d$ & $1$
\end{tabular}
& \hspace*{11pt}
\begin{tabular}{c|cccccc}
$\odot $ & $0$ & $a$ & $b$ & $c$ & $d$ & $1$ \\ \hline
$0$ & $0$ & $0$ & $0$ & $0$ & $0$ & $0$ \\
$a$ & $0$ & $a$ & $b$ & $d$ & $d$ & $a$ \\
$b$ & $0$ & $b$ & $b$ & $0$ & $0$ & $b$ \\
$c$ & $0$ & $d$ & $0$ & $d$ & $d$ & $c$ \\
$d$ & $0$ & $d$ & $0$ & $d$ & $d$ & $d$ \\
$1$ & $0$ & $a$ & $b$ & $c$ & $d$ & $1$
\end{tabular}
\end{tabular}\end{center}

The maximal filters of $A$ are $\{a,b,1\}$ and $\{a,c,d,1\}$, so ${\rm Rad}(A)=\{a,1\}$.  It is easy to verify that 
$B(A)=\{0,1\}$ and
\[B(A/{\rm Rad}(A))=A/{\rm Rad}(A)=\{0/{\rm Rad}(A),b/{\rm Rad}(A),c/{\rm Rad}(A), 1/{\rm Rad}(A)\}.\]
Since $|B(A/{\rm Rad}(A))|=4>2=|B(A)|$, $B(r_A)$ can not be surjective.
\end{example}

\section{Semilocal residuated lattices}\label{semiloc}

\noindent A residuated lattice is said to be {\em semilocal} iff it has only a finite number of maximal filters.

The trivial residuated lattice has no maximal filters, hence it is obviously semilocal. We shall consider only nontrivial semilocal residuated lattices. The class of semilocal residuated lattices includes finite residuated lattices as well as the local ones. It is easy to construct examples of semilocal residuated lattices that are not local: any finite direct product of $n\ge 2$ local residuated lattices has exactly $n$ maximal filters, by Theorem \ref{spec+max-finite-product}.

\begin{remark}
The class of semilocal residuated lattices is  a pseudo-variety, i.e., it is closed under finite direct products, homomorphic images and subalgebras. 
\end{remark}
\begin{proof}
Apply Theorem \ref{spec+max-finite-product}, Proposition \ref{hom}.(\ref{hom-cardinal-image}) and Proposition \ref{subalgebra-cardinal-filters}.\end{proof}

\bprop\label{prop-semilocal}
Let $A$ be a nontrivial residuated lattice and $F$ be a proper filter such that $F\se {\rm Rad}(A)$. Then
\be
\item $A$ is semilocal if and only if $A/F$ is semilocal;
\item\label{} if $A$ is a semilocal and ${\rm Max}(A)\!=\!\{M_1,\ldots, M_n\}$, then 
\[A/{\rm Rad}(A)\cong \prod_{i=1}^n A/M_i.\]
\ee
\eprop

\begin{proof}
\be
\item By Proposition \ref{Max-A-AF-RadA}.
\item Let us consider the canonical projection
\[\varphi: A/{\rm Rad}(A)\ra \prod_{i=1}^n A/M_i, \quad \varphi(a/{\rm Rad}(A))=(a/M_1,\ldots ,a/M_n).\]
For every $a,b\in A$, we get that 
$\varphi(a/{\rm Rad}(A))=\varphi(b/{\rm Rad}(A))$ iff $a/M_i=b/M_i$ for all $i\in \ol{1,n}$ iff $a\leftrightarrow b\in M_i$ for all $i\in \ol{1,n}$ iff $a\leftrightarrow b\in {\rm Rad}(A)$ iff $a/{\rm Rad}(A)=b/{\rm Rad}(A)$. Hence, $\varphi$ is well defined and injective. Since $M_1,\ldots, M_n$ are distinct maximal filters, it follows that $M_i\sau M_j=A$ for all $i\ne j$, so we can apply Proposition \ref{chinese} to get for every $a_1,\ldots, a_n\in A$ an $a\in A$ such that $a/M_i=a_i/M_i$ for all $i\in \ol{1,n}$, so $\varphi (a/{\rm Rad}(A))=(a_1/M_1,\ldots, a_n/M_n)$. Thus, we have proved that $\varphi$ is surjective too. It is easy to see that $\varphi$ is a morphism of residuated lattices, hence it is an isomorphism of residuated lattices.
\ee \end{proof}

\section{Maximal residuated lattices}\label{maximal}

\noindent Let $A$ be a residuated lattice, $I$ an index set, $\{a_i\}_{i\in I}\se A$ and $\{F_i\}_{i\in I}$ be a family of filters of $A$. We say that the family $\{(a_i,F_i)\}_{i\in I}$ has {\em finite intersection property} (abbreviated {\em f.i.p.}) iff the family of sets $\left\{a_i/F_i\right\}_{i\in I}$ has finite intersection property, i.e. the intersection of every finite subfamily is nonempty. Formally,  $\{(a_i,F_i)\}_{i\in I}$ has f.i.p. iff 
\bce
for any finite $J\se I$ there exists $x_J\in A$ with $x_J\equiv a_i(mod\, F_i)$ for all $i\in J$.
\ece
$A$ is said to be {\em maximal} iff whenever $\{(a_i,F_i)\}_{i\in I}$ is a family with f.i.p., there exists $x\in A$ such that $x\equiv a_i(mod\,\,F_i)$ for all $i\in I$.

In this section we prove the main result of the paper, the structure theorem for maximal residuated lattices with lifting Boolean center. Before this, we give some useful properties of maximal residuated lattices.

By the Chinese Remainder Theorem, the following lemma is immediate.

\blem\label{lem-maximal-1}
Let $A$ be a maximal residuated lattice. Then for any family $\{a_M\}_{M\in {\rm Max}(A)}$ of elements of $A$ there exists $x\in A$ such that $x\equiv a_M(mod\, M)$ for all $M\in {\rm Max}(A)$.
\elem

\bprop\label{maximal=>semilocal}
Any maximal residuated lattice is semilocal. 
\eprop

\begin{proof}
Let $A$ be a maximal residuated lattice, $F_M$ be the filter of $A$ defined in Lemma \ref{filtre-max-1} and define ${\cal F}:=\{(1,F_M)\}\cup \{(0,M)\mid M\in {\rm Max}(A)\}$. 

In order to prove that the family ${\cal F}$ has f.i.p., let us consider a finite subfamily $\{(1,F_M),(0,M_1),\ldots, (0,M_n)\}$ and apply Lemma \ref{lem-maximal-1} to get the existence of an $x\in A$ satisfying $x\equiv 0(mod\, M_i)$ for all $i\in \ol{1,n}$ and $x\equiv 1(mod\, M)$ for all $M\in {\rm Max}(A)-\{M_1,\ldots, M_n\}$. By Lemma \ref{prop-cong}.(\ref{aF=0Fsau1F}),(\ref{prop-cong-Fproper}) and Lemma \ref{filtre-max-1}, we get that $x\in \bigcap \{M\mid M\in {\rm Max}(A)-\{M_1,\ldots, M_n\} \}\se F_M$, so $x\equiv 1(mod\, F_M)$. Since $A$ is maximal and ${\cal F}$ has f.i.p., there exists $y\in A$ such that $y\equiv 1(mod\, F_M)$ and $y\equiv 0(mod\,M)$ for all $M\in {\rm Max}(A)$. Thus, $y\in F_M$ and $y\nin M$ for any maximal filter $M$ of $A$. It follows that ${\rm Max}(A)=\{M\in {\rm Max}(A)\mid y\nin M\}$, which is finite due to the fact that $y\in F_M$.\end{proof}

The converse of the above proposition does not hold: an example of a semilocal MV-algebra that is not maximal can be found in \cite[Proposition 8]{figele}.

\bprop\label{<e>-maximal}
Let $A$ be a maximal residuated lattice and $e\in B(A)$. Then ${\bf <e>}$ is also a maximal residuated lattice.
\eprop

\begin{proof}
Let ${\cal F}=\{(a_i, F_i)\}_{i\in I}$ be a family that has f.i.p. in ${\bf <e>}$. Since $a_i\in A$ and every filter $F_i$ of ${\bf <e>}$ is also a filter of $A$, it follows that ${\cal F}$ has f.i.p. in $A$ too. Apply now the fact that $A$ is maximal to get an $x\in A$ such that $x\equiv a_i(mod\,\,F_i)$ in $A$ for all $i\in I$. By Proposition \ref{B(A)-2}.(\ref{F-Finte}), it follows that $x\sau e\equiv a_i\sau e(mod\, F_i\,\cap <e>)$ in ${\bf <e>}$ for all $i\in I$. Since $F_i\cap <e>=F_i$ as $F_i\se <e>$ and $a_i\sau e=a_i$ as $a_i\in <e>$, we get that   $x\sau e\in<e>$ is such that $x\sau e\equiv  a_i(mod\,F_i)$ for all $i\in I$. Thus, ${\bf <e>}$ is maximal.\end{proof}

\bprop\label{finite-product-maximal}
The class of maximal residuated lattices is closed under finite direct products.
\eprop

\begin{proof}
Assume that $A=\displaystyle\prod_{i=1}^n A_i,$ where $n\ge 1$ (the case $n=0$ is trivial), where $A_1,\ldots, A_n$ are maximal residuated lattices. By Proposition \ref{finite-direct-product-1}, there exist $e_1,\ldots, e_n\in B(A)$ satisfying $e_1\si \ldots \si e_n=0$, $e_i\sau e_j=1$ for $i\ne j$ such that $A_i\cong <\bf {e_i}>$ for all $i\in \ol{1,n}$. 
Let $\{(a_k, F_k)\}_{k\in K}$ have f.i.p. in $A$ and apply Proposition \ref{B(A)-2}(\ref{F-Finte}) to get that  the family $\{(a_k\sau e_i, F_k\cap <e_i>)\}_{k\in K}$ has f.i.p. in ${\bf <e_i>}$ for any $i\in \ol{1,n}$. Since ${\bf <e_i>}$ is maximal, there exists $x_i\in <e_i>$ such that $x_i\equiv a_k\sau e_i(mod\,F_k\cap <e_i>)$ for all $k\in K$: in particular, $x_i\equiv a_k\sau e_i(mod\,F_k)$ in $A$ for all $k\in K$. Let $x=x_1\si\ldots \si x_n$. Then $x\in A$ is such that $x\equiv (a_k\sau e_1)\si \ldots \si (a_k\sau e_n)(mod\,F_k)\equiv a_k(mod\,F_k)$ for all $k\in K$, since, by Lemma \ref{bcomut}.(\ref{comutarea}), $(a_k\sau e_1)\si\ldots \si(a_k\sau e_n)=a_k\sau(e_1\si\ldots \si e_n)=a_k\sau 0=a_k$. Thus, $A$ is maximal.\end{proof}

\bthm\label{char-maximal}
Let $A$ be a nontrivial residuated lattice with lifting Boolean center. Then the following are equivalent:
\be
\item $A$ is maximal;
\item there are $n\in \N^{\star}$ and $e_1,\ldots, e_n\in B(A)$ such that $\ds A\cong \prod_{i=1}^n<{\bf e_i}>$ and $<{\bf e_i}>$ is a nontrivial maximal residuated lattice for  all $i=\ol{1,n}$.
\ee
\ethm
\begin{proof}
$(ii)\Ra (i)$ Apply Proposition \ref{finite-product-maximal}.\\
$(i)\Ra (ii)$ By Proposition \ref{maximal=>semilocal}, $A$ is semilocal, hence ${\rm Max}(A)=\{M_1,\ldots, M_n\}$ for some $n\in{\rm I\! N}^{*}$ and, moreover, $\ds A/{\rm Rad}(A)\cong\displaystyle\prod_{i=1}^n A/M_i$ (see Proposition \ref{prop-semilocal}). Apply Proposition \ref{finite-direct-product-1} to get $f_1,\ldots, f_n\in B(A/{\rm Rad}(A))$ such that $f_1\si \ldots \si f_n=0/{\rm Rad}(A)$, $f_i\sau f_j=1/{\rm Rad}(A)$ for $i\ne j$ and $\ds A/M_i\cong {\bf <f_i>}$ for all $i= \ol{1,n}$.  Since $A$ has lifting Boolean center, there exist $e_1,\ldots,e_n\in B(A)$ such that $f_i=e_i/{\rm Rad}(A)$ for all $i\in \ol{1,n}$. It follows that $(e_1\si\ldots \si e_n)/{\rm Rad}(A)=f_1\si\ldots \si f_n=0/{\rm Rad}(A)$, so $\neg \, (e_1\si\ldots \si e_n)\in B(A)\cap {\rm Rad}(A)=\{1\}$, by Lemma \ref
{prop-cong}.(\ref{aF=0Fsau1F}) and Proposition \ref{BA-prop}.(\ref{BA-RadA=1}). Thus, $\neg \, (e_1\si\ldots \si e_n)=1$, so $e_1\si\ldots \si e_n=\neg \, \neg \, (e_1\si\ldots \si e_n)=\neg \, 1=0$, by Proposition  \ref{BA-prop}.(\ref{BA-neg-e}). We get similarly that, for $i\ne j$, $(e_i\sau e_j)/{\rm Rad}(A)=f_i\sau f_j=1/{\rm Rad}(A)$, so $e_i\sau e_j=1$. Applying now Proposition \ref{finite-direct-product-2}, it follows that $A\cong \displaystyle\prod_{i=1}^n{\bf <e_i>}$. Moreover,  ${\bf <e_i>}$ is maximal for all $i\in \ol{1,n}$, by Proposition \ref{<e>-maximal}. Finally, ${\bf <e_i>}$ is nontrivial, since $A/M_{i}$ is nontrivial.\end{proof}

\bthm\label{main-theorem}
Let $A$ be a nontrivial residuated lattice with lifting Boolean center. If $A$ is maximal, then $A$ is isomorphic to a finite direct product of local residuated lattices, each of which is clearly a maximal residuated lattice.
\ethm
\begin{proof}
By Theorem \ref{char-maximal}, $A\cong\displaystyle\prod_{i=1}^n{\bf <e_i>}$, where $n=|{\rm Max}(A)\mid \ge 1$, $e_1,\ldots, e_n\in B(A)$ and ${\bf <e_i>}$ is nontrivial and maximal for all $i\in \ol{1,n}$. It remains to prove that ${\bf <e_i>}$ is local. Since ${\bf <e_i>}$ is nontrivial, $|{\rm Max}({\bf <e_i>})|\ge 1$ for any $i\in \ol{1,n}$. On the other hand, applying Theorem \ref{spec+max-finite-product}, we get that $n=|{\rm Max}(A)|=\displaystyle\sum_{i=1}^n|{\rm Max}({\bf <e_i>})|\ge n$. Thus, we must have $|{\rm Max}({\bf <e_i>})|=1$ for all $i\in \ol{1,n}$, that is ${\bf <e_i>}$ is local for all $i\in \ol{1,n}$.\end{proof}

\end{document}